\documentclass{amsart}

\newtheorem{thm}{Theorem}
\newtheorem{defn}[thm]{Definition}
\newtheorem{rem}[thm]{Remark}
\newtheorem{cor}[thm]{Corollary}
\newtheorem{prop}[thm]{Proposition}

\newcommand{\bl}{\begin{lem}}
\newcommand{\el}{\end{lem}}

\newcommand{\bml}{\begin{multline}}
\newcommand{\eml}{\end{multline}}

\newcommand{\beq}{\begin{equation}}
\newcommand{\eeq}{\end{equation}}
\newcommand{\bp}{\begin{prop}}
\newcommand{\ep}{\end{prop}}
\newcommand{\bd}{\begin{defn}}
\newcommand{\ed}{\end{defn}}
\newcommand{\pf}{\begin{proof}}
\newcommand{\epf}{\end{proof}}

\newcommand{\field}[1]{\ensuremath{\mathbb{#1}}}

\newcommand{\CC}{\field{C}}

\newcommand{\df}{=}
\newcommand{\NN}{\field{N}}

\newcommand{\PP}{\field{P}}

\newcommand{\RR}{\field{R}}

\newcommand{\ZZ}{\field{Z}}

\DeclareMathOperator{\hyp}{hyp}

\let\Re\relax
\DeclareMathOperator{\Re}{Re}

\newcommand{\mC}{\mathcal{C}}

\newcommand{\ord}{\mathrm{ord}}

\DeclareMathOperator{\en}{\mathnormal{\mathcal{E}(T)}}

\DeclareMathOperator{\gip}{\Gamma ^ \prime _\infty   }
\DeclareMathOperator{\gi}{\Gamma_\infty   }

\DeclareMathOperator{\vol}{vol}
\DeclareMathOperator{\R}{Re}

\DeclareMathOperator{\pc}{PSL(2,\CC)}

\DeclareMathOperator{\lp}{\Delta}

\newcommand{\Z}{\mathcal{Z}}

\usepackage{}

\theoremstyle{definition}

\theoremstyle{remark}

\numberwithin{equation}{section}

\begin{document}

\title[Superzeta functions]{Superzeta functions, regularized products, and the Selberg zeta function
on hyperbolic manifolds with cusps}


\author{Joshua S. Friedman}\thanks{First-named author: The views expressed in this article are the author's own and not those of the U.S. Merchant Marine Academy,
the Maritime Administration, the Department of Transportation, or the United States government.}
\address{Joshua S. Friedman \\
Department of Mathematics and Science \\
\textsc{United States Merchant Marine Academy} \\
300 Steamboat Road \\
Kings Point, NY 11024 \\
U.S.A.
}
\email{friedmanj@usmma.edu, CrownEagle@gmail.com}

\author{Jay Jorgenson}\thanks{Second-named author's research is supported
by NSF and PSC-CUNY grants.}
\address{Jay Jorgenson \\
Department of Mathematics \\
The City College of New York \\
Convent Avenue at 138th Street \\
New York, NY 10031
U.S.A.
}
\email{jjorgenson@mindspring.com}

\author{Lejla Smajlovi\'{c}}
\address{Lejla Smajlovi\'c \\
Department of Mathematics \\
University of Sarajevo\\
Zmaja od Bosne 35, 71 000 Sarajevo\\
Bosnia and Herzegovina}
\email{lejlas@pmf.unsa.ba}



%    The 2010 edition of the Mathematics Subject Classification is
%    now available.  If you are citing a classification from the
%    new scheme, use the following input coding instead.
%\subjclass[2010]{Primary }

\date{}

\begin{abstract}\noindent
Let $\Lambda = \{\lambda_{k}\}$ denote a sequence of complex numbers and
assume that that the counting function
$\#\{\lambda_{k} \in \Lambda ~|~ \vert \lambda_{k}\vert < T\} =O(T^{n})$ for some integer $n$.
From Hadamard's theorem, we can construct an entire function $f$ of order at most $n$ such that
$\Lambda$ is the divisor $f$.  In this article we prove, under reasonably general conditions, that the superzeta function
$\Z_{f}(s,z)$ associated to $\Lambda$ admits a meromorphic continuation.  Furthermore, we describe the relation
between the regularized product of the sequence $z-\Lambda$ and the function $f$ as constructed as a Weierstrass product.
In the case $f$ admits a Dirichlet series expansion in some right half-plane, we derive the
meromorphic continuation in $s$ of $\Z_{f}(s,z)$ as an integral transform of $f'/f$.  We apply these
results to obtain superzeta product evaluations of Selberg zeta function associated to finite
volume hyperbolic manifolds with cusps.
\end{abstract}
%\tableofcontents{}

\maketitle

\section{Introduction}

\subsection{Background discussion}
The zeta-regularized product approach to define the product of an infinite sequence of complex numbers is,
at this point, a well-accepted mathematical concept.  Let $\Lambda = \{\lambda_{k}\}$ denote a sequence of
complex numbers and assume that that the counting function $\#\{\lambda_{k} \in \Lambda ~|~ \vert \lambda_{k}\vert < T\} =O(T^{n})$
for some integer $n$.  For any complex number $s$ with $\text{\rm Re}(s) > n$, the function
$$
\zeta_{\Lambda}(s) = \sum\limits_{\lambda_{k} \in \Lambda, \lambda_{k} \neq 0}\lambda_{k}^{-s}
$$
is convergent.  Under additional conditions, $\zeta_{\Lambda}(s)$ admits a meromorphic continuation to a region
which includes $s=0$ and is holomorphic at $s=0$.  In such a case, one defines the zeta-regularized product
of $\Lambda$, which is denoted by $\det^{\ast}(\Lambda)$, to be $\text{\rm det}^{\ast}(\Lambda) = \exp(-\zeta_{\Lambda}'(0))$
where the prime denotes a derivative with respect to the variable $s$.  The use of the term \it determinant \rm stems from
example when $\Lambda$ is a finite sequence of non-zero numbers associated to a finite dimensional linear operator $T$ in which
case $\exp(-\zeta_{\Lambda}'(0))$ is, indeed, the determinant of $T$.  In general, the use of the asterisk signifies that
one considers only the non-zero elements of $\Lambda$.

As stated, the zeta-regularized product of $\Lambda$ requires that $\Lambda$ satisfies additional conditions
which are listed as general axioms in \cite{JL93b}.  One of the conditions is that the associated theta function
$\displaystyle \theta_{\Lambda}(t) = \sum_{\lambda_{k} \in \Lambda} e^{-\lambda_{k}t}$ is defined for real and positive
values of $t$ and admits an asymptotic expansion as $t$ approaches zero.  This condition allows one to compute $\zeta_{\Lambda}(s)$
as the Mellin transform of $\theta_{\Lambda}(t)$.  If $\Lambda$ is the sequence of eigenvalues of some Laplacian operator $\mathbf \Delta$
acting of a Hilbert space associated to compact Riemannian manifolds, then one views $\det^{\ast}(\Lambda)$ as the
determinant of $\mathbf \Delta$ and one writes $\det^{\ast}{\mathbf \Delta} = \det^{\ast}(\Lambda)$.

In \cite{RaySing73} the authors consider, among other ideas, the zeta-regularized product of the sequence of eigenvalues
associated to the Laplacians acting on smooth sections of flat line bundles over genus one curves, when equipped with
a unit volume flat metric, and higher genus Riemann surfaces, when equipped with hyperbolic metrics of constant negative
curvature.  In the genus one case, the evaluation of the determinant of the Laplacian amounts to Kronecker's second limit
formula, resulting in the evaluation of determinant of the Laplacian in terms of Riemann's theta function.  For higher
genus surfaces, the determinant of the Laplacian is expressed in terms of Selberg's zeta function.
Far-reaching applications of determinants of the Laplacian, and various sums known as analytic torsion, provide the underpinnings
of the analytic aspect of Arakelov theory.
%Space limitations in this article prevent us from providing even a
%brief overview of this fascinating and far-reaching program of study.

\subsection{Our main results}
If we continue to consider the setting of Laplacians on finite volume Riemann surfaces, we note that the zeta-regularized
approach to defining a determinant of the Laplacian does not extend to the setting when the surface admits cusps.  Articles
exist where the authors propose means to define determinants of the Laplacian in the non-compact setting by
regularizing the trace of the heat kernel in some sense; see, for example, \cite{JLund95}, \cite{JLund96}, \cite{JLund97} and \cite{Mull98}.
The purpose of the present article is to provide another approach to the problem.  The point of view we consider can be explained through
the following example.
Let $M$ denote a finite volume hyperbolic Riemann surface, and let $\mathbf \Delta$ denote the associated Laplacian which acts
on the space of smooth functions on $M$.  If $M$ is compact, the spectrum of $\mathbf \Delta$ consists of discrete eigenvalues.
If $M$ is not compact, then the corresponding spectrum has discrete eigenvalues, whose nature is the subject of the Phillips-Sarnak
philosophy \cite{PS92}, as well as a continuous spectrum with associated measure $d\mu$ which is absolutely continuous with respect to Lebesgue measure
on $\RR$.  Furthermore, one can write $d\mu (r) = \phi'/\phi (1/2+ir) dr$ where $dr$ denotes Lebesgue measure and $\phi(s)$ is a meromorphic function
of $s \in \CC$ but whose restriction to the line $\Re(s) =1/2$ appears in the evaluation of the spectral measure.  The set $\Lambda$ to be
constructed consists of the poles of the function $\phi(s)$ for all $s \in \CC$ together with additional points in $\CC$ associated
to the $L^{2}$ spectrum, appropriately computed in the parameter $r$.

Once the set $\Lambda$ has been determined, then the regularized product we define follows the superzeta function construction
which was systematically developed by Voros; see \cite{Voros1}, \cite{Voros2} and \cite{Voros3}.  In vague terms, which will be clarified
below, one defines $\displaystyle \Z_{f}(s,z)=\sum_{\lambda_{k}\in\Lambda }(z-\lambda_{k})^{-s}$ for $s \in \CC$ with $\text{\rm Re}(s)$ sufficiently large
and $z\in \CC$ suitably restricted.  The superzeta regularized product of $z-\Lambda$ is obtained by proving the meromorphic continuation
in $s$ and employing a special value at $s=0$.

The purpose of this paper is to establish general circumstances under which the proposed construction is valid.  In section 2 we prove that for
a general entire function $f$ of finite order, one can define the superzeta regularized product of its zeros.  Furthermore, we establish a
precise relation between the Hadamard product representation of $f$ and $\exp(-\Z_{f}'(0,z))$ where the prime denotes a derivative in $s$.  In
section 3 we present an example of the results from section 2 taking $f$ to be the Selberg zeta function associated to any co-finite Kleinian
group. In section 4 we generalize results of section 2, by assuming that $f$ is a finite order meromorphic function which admits a general Dirichlet series representation.  We then prove, in Theorem 4.1, that one has the meromorphic continuation of the corresponding superzeta function.  Section 5 yields an example of the setting of section 4
obtained by considering the Selberg zeta function associated to any finite volume hyperbolic manifold of arbitrary dimension.

To summarize, the main results of this article combine to prove the analytic continuation of the superzeta regularized product for any meromorphic function which admits a general Dirichlet series representation.  The result is more general than the so-called ``ladder theorem'' from
\cite{JL93a} which requires a functional equation.

Finally, let us note that circumstances may occur when a sequence $\Lambda$ may arise from an operator.  In \cite{FJS16}
we took $\Lambda$ to be the divisor of the Selberg zeta function for a finite volume, non-compact hyperbolic Riemann
surface.  In this case, it was discussed in \cite{FJS16} how $\Lambda$ can be viewed in terms of the Lax-Phillips
scattering operator in \cite{Lax-Phill76}, so then we obtain a construction of what could be viewed as the determinant of this operator.

\section{Zeta regularization of entire functions}

Let $\RR^{-} = (-\infty,0]$ be the non-positive real numbers. Let $\{y_{k}\}_{k\in \mathbb{N}}$ be the sequence of zeros
of an entire function $f$ of order $\kappa\geq1$, repeated with their multiplicities. Let
$$
X_f = \{z \in \CC~|~ (z-y_{k}) \notin \RR^{-}~\text{for all} ~ y_{k} \}.
$$
For $z \in X_f,$ and $s \in \CC$ (where convergent) consider the series
\begin{equation}
\Z_{f}(s,z)=\sum_{k=1}^{\infty }(z-y_{k})^{-s},  %\label{Zeta1}
\end{equation}
where the complex exponent is defined using the principal branch of the logarithm with $\arg z\in
\left( -\pi ,\pi \right) $ in the cut plane $\CC \setminus \RR^{-}. $

Since $f$ is of order $\kappa$, $\Z_{f}(s,z)$ converges absolutely for $\Re(s) > \kappa.$
The series $\Z_{f}(s,z)$ is called the zeta function associated to the zeros of $f $, or simply the \emph{superzeta} function of $f.$

If $\Z_{f}(s,z)$ has a meromorphic continuation which is regular at $s=0,$ we define, for $z\in X_f$ the \emph{zeta regularized product} associated to $f$ as
$$ D_{f}\left( z \right) = \exp\left( {-\frac{d}{ds}\left. \Z_{f}\left( s,z\right) \right|_{s=0}  } \right).$$

Let $m= \lfloor \kappa \rfloor$, where $\lfloor \kappa \rfloor$ denotes the largest integer less than or equal to $\kappa$. Hadamard's product formula allows us to write
\beq
f(z) = \Delta_{f}(z) = e^{g(z)} z^r \prod_{k=1}^\infty \left( \left(1-\frac{z}{y_k}     \right)\exp\left[ \frac{z}{y_k} +...+ \frac{z^m}{m{y_k}^m}   \right]    \right),
\eeq
where $g(z)$ is a polynomial of degree $m$ or less, $r\geq 0$ is the order of the eventual zero of $f$ at $z=0,$ and the other zeros $y_k$ are listed with multiplicity.
A simple calculation shows that when $z \in X_f,$
\beq \label{eqTripDer}
\Z_{f}(m+1,z)=\frac{(-1)^m}{m!}\left(\log \Delta _{f}\left( z\right) \right)
^{(m+1) }.
\eeq

The following proposition is due to Voros (\cite{Voros1}, \cite{Voros3}, \cite{VorosKnjiga}). We also give a different proof.

\begin{prop} \label{prop: Voros cont.}
Let $f$ be an entire function of order $\kappa \geq 1$, and for $k\in\NN,$ let $y_k$ be the sequence of zeros of $f.$ Let $\Delta_{f}(z)$
denote the Hadamard product representation of $f.$ Assume that for $n>m=\lfloor \kappa \rfloor$ we have the following asymptotic expansion:
\begin{equation} \label{defAE}
\log \Delta_{f}(z)= \sum_{j=0}^{m}\widetilde{a}_{j}z^{j}(\log z-H_j)+ \sum_{j=0}^{m} b_{j}z^{j}+\sum_{k=1}^{n-1}a_{k}z^{\mu _{k}} + h_n(z),
\end{equation}
where $H_0=0$, $H_j=\sum_{l=1}^j (1/l)$, for $j\geq 1$, $1>\mu _{1}>...>\mu _{n} \rightarrow -\infty $, and $h_n(z)$ is a sequence of holomorphic functions in the sector $\left\vert \arg z\right\vert <\theta <\pi, \quad (\theta >0)$ such that $h_n^{(j)}(z)=O(|z|^{\mu_n-j})$, as $\left\vert z\right\vert \rightarrow \infty $ in the above sector, for all integers $j \geq 0.$

Then, for all $z\in X_f,$ the superzeta function $\Z_{f}(s,z)$  has a meromorphic continuation to the half-plane $\Re(s)<\kappa$ which is regular at $s=0.$
Furthermore, the zeta regularized product $D_{f}\left( z\right) $ associated to $\Z_{f}(s,z)$ is related to $\Delta_{f}(z)$ through the formula
\begin{equation}
D_{f}(z)=e^{-(\sum_{j=0}^{m} b_{j}z^{j})}\Delta_{f}(z)  \label{D(z)}
\end{equation}
which also provides analytic continuation of $D_f(z)$ from $X_f$ to the whole complex $z-$plane.
\end{prop}

\begin{proof}

For any $z\in X_f$, the series
\beq
\Z_{f}(m+1, z+y)= \sum_{k=1}^{\infty }(z+y-y_{k})^{-(m+1)} \label{zeta2}
\eeq
converges absolutely and uniformly for  $y \in (0,\infty).$  Furthermore, application of \cite[Formula 3.194.3]{GR07}, with $\mu=m+1-s$, $\nu = m+1$ and $\beta=(z-y_k)^{-1}$ yields, for all $y_k,$
$$
\int\limits_{0}^{\infty}\frac{y^{m-s}\,dy }{(z+y-y_k)^{m+1} }=\frac{1}{m!}(z-y_k)^{-s}\Gamma(m+1-s)\Gamma(s).
$$
Absolute and uniform convergence of the series \eqref{zeta2} for $\Re(s)>\kappa$ implies that
\begin{align}\label{Zintegralnareprezentacija}
\Z_{f}(s,z)&=\frac{
m!}{\Gamma(m+1-s)\Gamma(s)}\int_{0}^{\infty }\Z_{f}(m+1,z+y)y^{m-s}dy \\ &= \frac{(-1)^m}{\Gamma(m+1-s)\Gamma(s)}\int_{0}^{\infty } \left(\log \Delta _{f}\left( z+y\right) \right)
^{(m+1) } y^{m-s}dy \notag,
\end{align}
for $\kappa< \Re(s) <m+1.$

Next, we use \eqref{Zintegralnareprezentacija} together with \eqref{defAE} in order to get the meromorphic continuation of $Z_{f}(s,z)$ to the half plane Re$(s)<m+1$.
We start with \eqref{eqTripDer} and differentiate Equation~(\ref{defAE}) $(m+1)$ times to get
\begin{align*}
\left(\log \Delta_f(z+y)\right)
^{(m+1) }& =  \sum_{j=0}^{m}   \frac{(-1)^{m-j} j!(m-j)! \widetilde{a}_j}{(z+y)^{(m+1-j)}}+\sum_{k=1}^{n-1}\frac{a_{k}\mu_k(\mu_k-1)\cdot ... \cdot(\mu_k-m)}{(z+y)^{m+1-\mu _{k}}}\\ &+ h_n^{(m+1)}(z+y),
\end{align*}
for any $n>m$.

Since $\mu_k \searrow -\infty$, for an arbitrary $\mu<0$ there exists $k_0$ such that $\mu_k \leq \mu$ for all $k\geq k_0$, hence we may write
\begin{align*}
\left(\log \Delta _{f}\left( z+y\right) \right)
^{(m+1) } y^{m+1} &= y^{m+1}\left(\sum_{j=0}^{m}   \frac{(-1)^{m-j} j!(m-j)! \widetilde{a}_j}{(z+y)^{m+1-j}} \right.\\ &+ \left.\sum_{k=1}^{n-1}\frac{a_{k}\mu_k(\mu_k-1)\cdot ... \cdot(\mu_k-m)}{(z+y)^{m+1-\mu _{k}}}\right) + g_{\mu}(z+y),
\end{align*}
where $g_{\mu}(z+y)=y^{m+1}h_{k_0}^{(m+1)}(z+y).$

Note that
\beq g_{\mu}(z+y) = O(y^{\mu}) \quad \text{as $y \to \infty,$} \quad \text{and} \quad g_{\mu}(z+y) = O(y^{m+1}) \text{ as $y \searrow 0$}. \label{eqDecayG}
\eeq
Application of \cite[Formula 3.194.3]{GR07} yields
\begin{multline}\label{z int for continuation}
\frac{(-1)^m}{\Gamma(m+1-s)\Gamma(s)}\int_{0}^{\infty }\left(\log \Delta _{f}\left( z+y\right) \right)
^{(m+1) }y^{m-s}dy= \sum_{j=0}^{m} (-1)^j j!\widetilde{a}_j \frac{\Gamma(s-j)}{\Gamma(s)} z^{j-s}\\
- \sum_{k=1}^{k_0-1}a_{k}\frac{\Gamma(s-\mu_k)}{\Gamma(s) \Gamma(-\mu_k)} z^{\mu_k-s}
 +\frac{(-1)^m}{\Gamma(m+1-s)\Gamma(s)} \int_{0}^{\infty }g_{\mu}(z+y)y^{-s-1}dy.
\end{multline}
The integral on the right hand side of \eqref{z int for continuation} is the Mellin transform of the function $g_{\mu}.$ By \eqref{eqDecayG} this integral represents a holomorphic function in $s$ for all $s$ in the half strip $\mu < \mathrm{Re}(s)<m+1$. The other terms on the right hand side of \eqref{z int for continuation} are meromorphic in $s$, hence, by \eqref{Zintegralnareprezentacija}, the right-hand side of \eqref{z int for continuation} provides meromorphic continuation of integral $\int_{0}^{\infty }Z_{f}(m+1,z+y)y^{m-s}dy$ from the strip $\kappa<\mathrm{Re}(s)<m+1$ to the strip $\mu < \mathrm{Re}(s)<m+1$. Since $\mu<0$ was chosen arbitrarily, we can let $\mu \to -\infty$ and obtain the meromorphic continuation of this integral to the half plane Re$(s)<m+1.$

Formula \eqref{z int for continuation}, together with \eqref{Zintegralnareprezentacija}, now yields the following representation of $\Z_f(s,z)$, for an arbitrary, fixed $z\in X_f,$ valid in the half plane $\mathrm{Re}(s)<m+1$:
\begin{multline} \label{Zf repres}
\Z_f(s,z)= \sum_{j=0}^{m} (-1)^j j!\widetilde{a}_j \frac{\Gamma(s-j)}{\Gamma(s)} z^{j-s} - \sum_{k=1}^{k_0-1}a_{k}\frac{\Gamma(s-\mu_k)}{\Gamma(s) \Gamma(-\mu_k)} z^{\mu_k-s}\\
 +\frac{(-1)^m}{\Gamma(m+1-s)\Gamma(s)} \int_{0}^{\infty }h_{k_0}^{(m+1)}(z+y)y^{m-s}dy.
\end{multline}
From the decay properties of $h_{k_0}^{(m+1)}(z+y),$ it follows that $\Z_f(s,z)$ is holomorphic at $s=0.$  Furthermore since  $\tfrac{1}{\Gamma(s)}$ has a zero at $s=0,$  the derivative of the last term in \eqref{Zf repres} is equal to
\begin{align*}
\left( \left. \frac{d}{ds}\frac{1}{\Gamma(s)}\right|_{s=0}\right) \frac{(-1)^m}{\Gamma(m+1)}\int_{0}^{\infty }h_{k_0}^{(m+1)}(z+y)y^{m}dy &=  \frac{(-1)^m}{m!}\int_{0}^{\infty }h_{k_0}^{(m+1)}(z+y)y^{m}dy\\& =- h_{k_0}(z),
\end{align*}
where the last equality is obtained from integration by parts $m$ times, and  using the decay of $h_{k_0}(z+y)$ and its derivatives as $y\to+\infty$, for $\mu_{k_0}<0.$ Moreover, since
$$
\left.\frac{d}{ds}\frac{\Gamma(s-\mu_k)}{\Gamma(s)} \right|_{s=0}=  \Gamma(-\mu_k)\cdot \left. \frac{d}{ds}\frac{1}{\Gamma(s)}\right|_{s=0} =\Gamma(-\mu_k),
$$
elementary computations yield that
$$
\left.\frac{d}{ds}\left(\sum_{k=1}^{k_0-1}a_{k}\frac{\Gamma(s-\mu_k)}{\Gamma(s) \Gamma(-\mu_k)} z^{\mu_k-s} \right) \right|_{s=0} = \sum_{k=1}^{k_0-1}a_{k}z^{\mu_k}.
$$
This shows that
\beq \label{deriv}
-\left. \frac{d}{ds} \Z_f(s,z) \right|_{s=0}=-\left. \frac{d}{ds} \left(\sum_{j=0}^{m} (-1)^j j!\widetilde{a}_j \frac{\Gamma(s-j)}{\Gamma(s)} z^{j-s} \right) \right|_{s=0}+\sum_{k=1}^{k_0-1}a_{k}z^{\mu _{k}} + h_{k_0}(z),
\eeq
for $z$ in the sector $\left\vert \arg z\right\vert <\theta <\pi $, $(\theta >0)$.

Now, for $j\in\{1,...,m\}$ one has
$$
\left. \frac{d}{ds} \left( \frac{\Gamma(s-j)}{\Gamma(s)} z^{j-s} \right) \right|_{s=0} = (-\log z) z^j \prod_{k=1}^j (0-k)^{-1} +  z^j \left. \frac{d}{ds}\frac{\Gamma(s-j)}{\Gamma(s)}  \right|_{s=0}.
$$
A straightforward computation shows that
$$
\left. \frac{d}{ds}\frac{\Gamma(s-j)}{\Gamma(s)}  \right|_{s=0} = \left. \frac{d}{ds}\prod_{k=1}^j (s-k)^{-1} \right|_{s=0}= \frac{(-1)^j}{j!}\sum_{k=1}^{j}\frac{1}{k}= H_j \cdot \frac{(-1)^j}{j!}.
$$
Therefore,
$$
\left. \frac{d}{ds} \left((-1)^j j!\widetilde{a}_j \frac{\Gamma(s-j)}{\Gamma(s)} z^{j-s} \right) \right|_{s=0}= \widetilde{a}_j (-\log z) z^j + \widetilde{a}_j z^j H_j,
 $$
for $j\in\{1,...,m\}$ and hence
$$
-\left. \frac{d}{ds} \left(\sum_{j=0}^{m} (-1)^j j!\widetilde{a}_j \frac{\Gamma(s-j)}{\Gamma(s)} z^{j-s} \right) \right|_{s=0}= \sum_{j=0}^{m}\widetilde{a}_j z^j (\log z - H_j).
$$

Now, \eqref{D(z)} follows from equation \eqref{deriv} and uniqueness of analytic continuation.
\end{proof}

\section{Superzeta functions constructed from confinite Kleinian groups}

Let  $\Gamma $ be a cofinite  Kleinian group.  Suppose $T\in\Gamma$ is loxodromic. Then $T$ is conjugate in $\pc$ to a unique element of the form $\left(\smallmatrix a(T)& 0\\ 0& a(T)^{-1} \endsmallmatrix\right)$
%$$
%D(T)=
%\left(\begin{array}{cc}
%a(T) & 0\\
%0 & a(T)^{-1}
%\end{array}\right) $$
such that $a(T)\in\CC$ has $|a(T)|>1$.  Let $N(T) \df |a(T)|^{2},$ and  let  $\mC(T) $ denote  the centralizer of $T$ in $\Gamma.$  There exists a (primitive)  loxodromic element $T_0,$ and a finite cyclic elliptic subgroup  $\en$ of order $m(T), $ generated by an element $E_T $   such that
$\mC(T) = \langle T_{0} \rangle \times \en, $ where $\langle  T_{0} \rangle = \{\, T_{0}^{n} ~ | ~ n \in\ZZ ~ \}. $

%Let $\mathfrak{t}_1,\dots, \mathfrak{t}_n, $ and $\mathtt{t'_1},\dots,  \mathtt{t'_n}$ denote the eigenvalues of $\chi(T_0)$ and $\chi(E_T)$ respectively.

The elliptic element $ E_T$ is conjugate in $\pc$ to an element of the form  $\left(\smallmatrix \zeta(T_0)& 0\\ 0& \zeta(T_0)^{-1} \endsmallmatrix\right)$
%$$\left(\begin{array}{cc}
%\zeta(T_0) & 0 \\
%0 & \zeta(T_0)^{-1}
%\end{array}\right), $$
where here $\zeta(T_0)$ is a primitive $2m(T)$-th root of unity.

\bd \label{defSZ}
For $\R(s)>1 $ the Selberg zeta-function $Z(s)$ is defined by
$$
Z(s) \df \prod_{ \{T_0 \} \in \mathcal{R}} ~ \prod_{  \substack{ l,k \geq 0 \\  c(T,l,k)=1   } } \left( 1-a(T_0)^{-2k} \overline{ a(T_0) ^{-2l}} N(T_0)^{-s - 1}    \right).
$$
Here the product with respect to $T_0$ extends over a maximal reduced system $\mathcal{R} $ of $\Gamma$-conjugacy classes of primitive loxodromic elements of $\Gamma.$ The system  $\mathcal{R} $ is called reduced if no two of its elements have representatives with the same centralizer\footnote{See  \cite{Elstrodt} section 5.4 for more details}.  The function  $c(T,l,k)$ is defined by
$c(T,l,k)= \zeta(T_0)^{2l}  \zeta(T_0)^{-2k}.$
\ed

%For $\R(s)>1,$
%$$\frac{d}{ds} \log Z(s)  = \sum_{ \{ T \}\LOX}  \frac{\log %N(T_{0})}{m(T)|a(T)-a(T)^{-1}|^{2}}N(T)^{-s}. $$

Let $\zeta \in \PP^1$ be cusp of $\Gamma,$
$
\Gamma_{\zeta} = \{~ \gamma \in \Gamma ~ | ~ \gamma \zeta = \zeta ~\},
$
and  let $ \Gamma_\zeta^\prime $ be the maximal torsion-free parabolic subgroup of $\Gamma_\zeta$.
The possible values for the index of $[\Gamma_\zeta:\Gamma_\zeta^\prime]$ are 1,2,3,4, and 6.

\begin{prop}[\cite{Frid05}]
Let  $\Gamma$ be cofinite with one class of cusp at $\infty$. If $[\gi:\gip] = 1$ or $[\gi:\gip] = 2,$  then  $Z(s)$ is a meromorphic function. Furthermore, the zeros of the Selberg zeta-function are:
\begin{enumerate}
\item zeros at the points $\pm s_j$ on the line $\R(s)=0$ and on the interval $[-1,1].$  Each point $s_j$ is related to an eigenvalue $ \lambda_j $ of the discrete spectrum of (the self-adjoint extension) of Laplacian $-\lp $ by $1-s_j^2 = \lambda_j.$  The multiplicity of each $s_j > 0,$ $m(s_j),$ is equal to the multiplicity of the corresponding eigenvalue $m(\lambda_j).$  If $-s_j < 0$ happens to also  be a zero of $\phi(s)$ of multiplicity $q(-s_j)$ then $m(-s_j) = m(\lambda_j) - q(-s_j).$ Here $\phi(s)$ is determinant of the automorphic scattering matrix (in this case, it is a scalar since we assume the number of cusps is one).

\item zeros at the points $\rho_j, $ that are poles of $\phi(s),$ which lie in the half-plane $\R(s) < 0. $ The multiplicity of each $\rho_j $ is equal its multiplicity, as a pole, of $\phi(s).$

\item if $[\gi:\gip] = 1$ then zeros at the points $s=\ZZ_{<0}$ with multiplicity 1; at the point $s=0$ it could be a pole or a zero, with multiplicity $\frac{1}{2}(\phi(0) - 1) + 2m(1),$ where $m(1)$ denotes the multiplicity of the possible eigenvalue $\lambda = 1.$

\item if $[\gi:\gip] = 2$ then zeros at the points $s=-1,-3,-5,\dots,$ with multiplicity 1; at the point $s=0$ it could be a pole or a zero with multiplicity $\frac{1}{2}(\phi(0) - 1) + 2m(1).$
\end{enumerate}
\end{prop}

Let $\left(\smallmatrix *& *\\ c_n & d \endsmallmatrix\right) \in \mathcal{R} $ denote the representatives of $\Gamma_\infty' \setminus ~\Gamma~/~\Gamma_\infty'.$
For $\R(s) > 1,$ the scattering matrix $\phi(s)$ can be written as a Dirichlet series
 $$ \phi(s) = \frac{\pi}{[\gi:\gip]} \frac{m(c_0)}{|\mathcal{P}|^s}|c_0|^{-2-2s}(1+ \sum_{n\geq 1} q_n^{-2-2s}),$$
where $q_n=|c_n/c_0| > 1, $  $|\mathcal{P}|$ is the co-area of the lattice associated to $\Gamma_\infty,$  $c_0$ is the minimal (non-zero) modulus lower left entry for $\Gamma,$ and $m(c_0)$ counts the number of times $c_0$ is present as a representative in $\mathcal{R}.$ See \cite[p. 111 and p.234 Eq. 1.10]{Elstrodt}) for more details.

\subsection{The case $[\gi:\gip] = 1$}
When $[\gi:\gip] = 1,$ we define $Z_1^+(s) = Z(s)/\Gamma(s),$ and let $\Z_1^+(s,z)$ be its superzeta function.  From the asymptotic version of Stirling's formula and from the product expansion of $Z(s),$ it follows that as $\Re(s) \rightarrow +\infty,$ we have that
$$
\log{Z_1^+(s)} \sim \frac{1}{2}\left(\log{s} - 0 \right) - 1 s\left(\log{s} - 1 \right) - \frac{1}{2}\log{2\pi}.
$$
Hence applying Proposition~\ref{prop: Voros cont.} to the finite-order entire function $Z_1^+(s),$ we obtain
the regularized product of the zeros of $Z_1^+(s)$ as $D_{Z_1^+}(s) = \sqrt{2\pi}Z_1^+(s). $

Next let $Z_1^-(s) = \phi(s)Z_1^+(s).$ Hence it follows that as $\Re(s) \rightarrow +\infty,$
$$
\log{Z_1^-(s)} \sim \frac{1}{2}\left(\log{s} - 0 \right) - 1 s\left(\log{s} - 1 \right) - \frac{1}{2}\log{2\pi} + \log\left(\frac{\pi m(c_0)}{|c_0|^2}\right) - s \log(|c_0|^2 |\mathcal{P}|).
$$
Once again, applying Proposition~\ref{prop: Voros cont.} we obtain that
$$
D_{Z_1^-}(s) = \frac{|c_0|^{2s} |\mathcal{P}|^{s}|c_0|^2}{\pi m(c_0)} \sqrt{2\pi} Z_1^-(s).
$$
Recalling that $Z_1^-(s) = \phi(s)Z_1^+(s)$  we obtain a formula for $\phi(s)$ as a quotient of regularized determinants,
$$ \phi(s) = \frac{D_{Z_1^-}(s)}{D_{Z_1^+}(s)}\frac{\pi m(c_0)}{|c_0|^{2s+2} |\mathcal{P}|^{s}}$$
from which one can obtain a formula for the central value of $\phi(s)$; see \cite{FJS16B}.
%In the Fuchsian case, the above formula was used in \cite{FJS16B} to evaluate $\phi(s)$ at its central value.

\subsection{The case $[\gi:\gip] = 2$}
When $[\gi:\gip] = 2,$ (this is the case e.g. for the Picard group) define
$$Z_2^+(s) = \frac{s Z(s)}{\Gamma\left(\frac{s-1}{2}\right)},$$ and let $\Z_2^+(s,z)$ be its superzeta function.

The asymptotic expansion of $\log{\Gamma(\frac{s-1}{2})}$ can be computed as
$$\log{\Gamma(\frac{s-1}{2})} \sim \frac{1}{2}s\log{s} -\frac{s}{2} - \frac{1}{2}s\log{2} - \log{s} + \frac{1}{2}\log{2\pi} + \log{2}.$$
Hence from the product expansion of $Z(s),$ it follows that as $\Re(s) \rightarrow +\infty,$
$$\log{Z_2^+(s)} \sim 2\log{s}  -\frac{1}{2}s\log{s} +\frac{s}{2} + \frac{1}{2}s\log{2}  - \frac{1}{2}\log{2\pi} - \log{2}  $$
and can be rewritten in the form required by Proposition~\ref{prop: Voros cont.}
$$\log{Z_2^+(s)} \sim 2(\log{s} - 0) -\frac{1}{2}s(\log{s} -1) + \frac{\log{2}}{2} s - \frac{1}{2}\log{2\pi} - \log{2}.  $$
Hence
$$D_{Z_2^+}(s) =  \sqrt{\pi}\, 2^{\left( \frac{3-s}{2} \right)}   Z_2^+(s).$$

If we define $Z_2^-(s) = \phi(s) Z_2^+(s)$ then we obtain
$$D_{Z_2^-}(s) =   \frac{|c_0|^{2s} |\mathcal{P}|^{s}|c_0|^2}{\pi m(c_0)} \sqrt{\pi}\, 2^{\left( \frac{3-s}{2} \right)} Z_2^-(s).$$

Finally we obtain
$$ \phi(s) = \frac{D_{Z_2^-}(s)}{D_{Z_2^+}(s)}\frac{\pi m(c_0)}{|c_0|^{2s+2} |\mathcal{P}|^{s}}. $$

\section{Zeta regularization of zeta-type functions through integral representation}

In this section we assume that $f$ is a meromorphic function of finite order $\kappa$ such that $\log f$ possesses a representation as a generalized Dirichlet series
\beq \label{Dirichlet series rep}
\log{f(z)}= \sum_{n=1}^{\infty} \frac{c_n}{q_n^z}
\eeq
converging absolutely and uniformly in any half plane $\Re(s) \geq \sigma + \epsilon$, for any $\epsilon >0$. Here $\{q_n\}$ denotes an increasing sequence of real numbers with $q_1> 1$ and $c_n$ are complex numbers. For such a function $f$ we say it is a \emph{zeta-type function}.

We denote by $N_f$ the set of zeros of $f$ and by $P_f$ the set of poles of $f$ and define
$$X_f = \{z \in \CC~|~ (z-z_{k}) \notin \RR^{-}~\text{for all} ~ z_{k} \in N_f\cup P_f \}. $$
For $z \in X_f,$ and $s \in \CC$ with $\Re(s) >\kappa$ consider the function
\begin{equation}
\Z_{f}(s,z)=\sum_{\rho\in N_f}\ord(\rho) (z-\rho)^{-s} - \sum_{\rho\in P_f}\ord(\rho) (z-\rho)^{-s}=\Z^\text{N}_{f}(s,z) - \Z^\text{P}_{f}(s,z) ,  \label{Zeta1}
\end{equation}
where $\ord(\rho)$ denotes the order of a zero or pole. We call $\Z_{f}(s,z)$ the \emph{superzeta function} associated to the meromorphic function $f$. Note that both series in \eqref{Zeta1} are absolutely convergent for $\Re(s)>\kappa$.

Denote by $W^\mathrm{N}_f$ and $W^\mathrm{P}_f$ Weierstrass products of order $m=\lfloor \kappa \rfloor$ associated to set of zeros and set of poles of $f$, respectively. Then, a straightforward computation shows that for $z\in X_f$
$$
\Z_{f}(m+1,z) = \frac{(-1)^m}{m!}\left( (\log W^\mathrm{N}_f(z))^{(m+1)} - (\log W^\mathrm{P}_f (z))^{(m+1)}\right),
$$
and, moreover
$$
(\log f(z))^{(m+1)} = (\log W^\mathrm{N}_f(z))^{(m+1)} - (\log W^\mathrm{P}_f (z))^{(m+1)}.
$$

We have the following theorem.

\begin{thm} \label{repZ1Thm} Let $f$ be a meromorphic function of zeta-type, of finite order $\kappa$. Fix $z\in X_f. $ The superzeta
function $\Z_{f}(s,z)$ defined by \eqref{Zeta1} has a holomorphic continuation to all $s \in \CC,$ through the equation
\begin{equation} \label{Z1rep}
\Z_{f}(s,z)=\frac{\sin \pi s}{\pi }\int_{0}^{\infty }\left( \frac{f^{\prime }}{f}%
(z+y)\right) y^{-s}dy.
\end{equation}
\end{thm}

\begin{proof}
For $ z\in X_f $ and  $\kappa < \R(s) < m+1,$ we apply equation~\eqref{eqTripDer} to $\Z_{f}^\text{N}(s,z)$ and $\Z_{f}^\text{P}(s,z)$ and proceed analogously as in the proof of \eqref{Zintegralnareprezentacija} to get
\begin{gather}
\Z_{f}(s,z) = \Z^\text{N}_{f}(s,z) - \Z^\text{P}_{f}(s,z) \notag \\=\frac{
(-1)^m}{\Gamma(m+1-s)\Gamma(s)}\int_{0}^{\infty }\left( (\log W^\mathrm{N}_f(z+y))^{(m+1)} - (\log W^\mathrm{P}_f (z+y))^{(m+1)}\right) y^{m-s}dy   \notag \\
=\frac{
(-1)^m}{\Gamma(m+1-s)\Gamma(s)}\int_{0}^{\infty }\left( \log f\left(
z+y\right) \right) ^{(m+1) }y^{m-s}dy. \label{zeta B plus G_1}
\end{gather}

For a fixed $z\in X_f,$ generalized Dirichlet series representation \eqref{Dirichlet series rep}  yields the asymptotic behavior
\begin{equation*}
\left( \log F\left( z+y\right) \right) ^{(l) }=O\left(\frac{1}{y^n}\right)%
\text{, \ as \ }y\rightarrow \infty, \text{   for all positive integers  }  n
\end{equation*}%
and
\begin{equation*}
\left( \log F\left( z+y\right) \right) ^{ (l) }=O(1)\text{, \
as \ } y  \searrow 0
\end{equation*}%
for any $1\leq l \leq m$.
Therefore, for $m-1< \Re(s) < m$ we may integrate by parts in \eqref{zeta B plus G_1} and obtain
\begin{gather} \label{ZplusGdrugo}
\Z_{f}(s,z) = \frac{(-1)^m}{\Gamma(m-s)(m-s)\Gamma(s)} \int_{0}^{\infty }y^{m-s}d\left( \left(
\log f\left( z+y\right) \right) ^{(m) }\right)= \\ \frac{(-1)^{m-1}}{\Gamma(m-s)\Gamma(s)}\int_{0}^{\infty }\left(
\log f\left( z+y\right) \right) ^{(m) }y^{m-1-s}dy \notag
\end{gather}%
Integrating by parts in \eqref{ZplusGdrugo} for $m-2< \Re(s)<m-1$ (in case when $m\geq 2$) we obtain
$$
\Z_{f}(s,z) = \frac{(-1)^{m-2}}{\Gamma(m-1-s)\Gamma(s)}\int_{0}^{\infty }\left(
\log f\left( z+y\right) \right) ^{(m-1) }y^{m-2-s}dy.
$$
Proceeding inductively in $m$, we deduce that
\beq \label{main fla for zf}
\Z_{f}(s,z) = \frac{1}{\Gamma(1-s)\Gamma(s)}\int_{0}^{\infty }\left(
\log f\left( z+y\right) \right) ' y^{-s}dy = \frac{\sin \pi s}{\pi }\int_{0}^{\infty }\left( \frac{f^{\prime }}{f}%
(z+y)\right) y^{-s}dy,
\eeq
for $0<\Re(s)<1$.

First, we claim that for a fixed $z\in X_f$, the integral
$$
I(s,z)=\int_{0}^{\infty }\left( \frac{f^{\prime }}{f}%
(z+y)\right) y^{-s}dy
$$
which appears on the right hand side of \eqref{main fla for zf} is actually holomorphic function in the half-plane $\Re(s)<1$. To see this, let $\mu \leq 0$ be arbitrary. Decay properties of $(\log f(z+y))'$, as $y\to +\infty$ with $n> -\mu +2$, yield that $(\log f(z+y))' y^{-s} = O(y^{-2})$,  as $y\to +\infty$, for all $s$ such that $\mu <\Re(s)\leq 0$. Moreover, the bound $\frac{f^{\prime }(z+y)}{f(z+y)}%
=O(1)$, for fixed $z\in X_f$ implies that $(\log f(z+y))' y^{-s} = O(1)$,  as $y\to 0$, for all $s$ in the half plane $\Re(s)\leq 0$. This shows that for $z\in X_f$ the integral $I(s,z)$ is absolutely convergent in the strip $\mu <\Re(s)\leq 0$, hence represents a holomorphic function for all $s$ in that strip. Since $\mu \leq 0$ was arbitrarily chosen, we have proved that $I(s,z)$, for $z\in X_f,$ is holomorphic function in the half plane $\Re(s)\leq 0$.

Next, we claim that $I(s,z)$, for $z\in X_f,$ can be meromorphically continued to the half-plane $\Re(s)>0$
with simple poles at the points $s=1,2,...$ and corresponding residues
\begin{equation}
\mathrm{Res}_{s=n}I(s,z)=-\frac{1}{(n-1)!}(\log
f(z))^{(n)}.
\end{equation}
Since the function $\sin(\pi s)$ has simple zeros at points $s=1,2,...$ this would prove that $\Z_{f}(s,z)$, for $z\in X_f$ is actually an entire function of $s$ and the proof would be complete.

Let $\mu>0$ be arbitrary, put $n=\lfloor \mu\rfloor$ to be the integer part of $\mu$ and let $\delta>0$ (depending upon $z\in X_f$ and $\mu$) be such that for $y\in (0,\delta)$ we have the Taylor series expansion
$$
(\log f(z+y))'=\sum_{j=1}^{n}\frac{(\log f(z))^{(j)}}{(j-1)!}y^{j-1} + R_1(z,y),
$$
where $R_1(z,y)=O(y^n)$, as $y\to 0$. Then, for $0<\Re(s)<1$ we may write
$$
I(s,z)= \sum_{j=1}^{n}\frac{(\log f(z))^{(j)}}{(j-1)!}\frac{\delta^{j-s}}{j-s} + \int_{0}^{\delta} R_1(z,y) y^{-s}dy + \int_{\delta}^{\infty }\left( \frac{f^{\prime }}{f}
(z+y)\right) y^{-s}dy.
$$
The bound on $R_1(z,y)$ and the growth of $(\log f(z+y))'$ as $y\to\infty$ imply that the last two integrals are holomorphic functions of $s$ for $\Re(s)\in(0,\mu)$. The first sum is meromorphic in $s$, for $\Re(s)\in(0,\mu)$, with simple poles at $s=j$, $j\in\{1,...,n\}$ and residues equal to $-(\log
f(z))^{(j)}/(j-1)!$. Since $\mu>0$ is arbitrary, this proves the claim and completes the proof.

\end{proof}

\begin{rem}  \label{remark} \rm
We often split a superzeta function $\Z_f(s,z)$ up as
$$
\Z_f(s,z)= \Z^\text{NT}_{f}(s,z) + \Z^\text{T}_{f}(s,z)-\Z^\text{P}_{f}(s,z),
$$
where $\Z^\text{NT}_{f}(s,z)$, $\Z^\text{T}_{f}(s,z)$ and $\Z^\text{P}_{f}(s,z)$ denote the superzeta functions associated to the non-trivial zeros, the trivial zeros, and the poles respectively.  Often $\Z^\text{T}_{f}(s,z)$ and $\Z^\text{P}_{f}(s,z)$ can be expressed in terms of some special functions, namely, the Hurwitz zeta functions, as in \cite{FJS16}. Upon applying Theorem \ref{repZ1Thm} one could deduce a relation giving a meromorphic continuation of the superzeta function associated to the non-trivial zeros of a zeta-type function. Moreover, by differentiating the equation \eqref{Z1rep} with respect to the $s-$variable one may directly compute the regularized determinant $D_f(z)$ associated to the superzeta function constructed over the set of non-trivial zeros.

The advantage of the full asymptotic expansion \eqref{defAE} is that it yields the coefficients in the regularized determinant expression, which can not immediately be derived from Theorem \ref{repZ1Thm}.
\end{rem}

\section{Superzeta functions constructed over non-trivial zeros of the Selberg zeta function on hyperbolic manifolds with cusps}

In this section we show how to apply Theorem \ref{repZ1Thm} to derive the meromorphic continuation of superzeta function constructed over the set of non-trivial zeros of the Selberg zeta function on a finite volume hyperbolic manifold with cusps. We also derive an expression for the regularized determinant associated to this superzeta function.

We will follow the notation from \cite{GonPark}, where the exact divisor of the Selberg zeta function is given.

Consider real hyperbolic space of dimension $d$, which can be realized as the symmetric space $G/K$, where $G=SO_0(d,1)$ and $K=SO(d)$. By $\Gamma$, we denote a discrete, torsion-free subgroup of $G$ such that $\vol(\Gamma\backslash G)$ is finite, where the volume is computed with respect to the standard normalized Haar measure on $G$. We assume that the metric induced from the Haar measure is such that the space $G/K$ has constant curvature equal to $-1$ and that the group $\Gamma$ satisfies the conditions $\Gamma_P=\Gamma \cap P = \Gamma \cap N(P)$ for  $P\in\mathfrak{P}_{\Gamma}$, which ensure that the resulting manifold $X_{\Gamma}= \Gamma \backslash G /K$ is a non-compact hyperbolic manifold with cusps. Here, $\mathfrak{P}_{\Gamma}$ is the set of all $\Gamma-$conjugacy classes of $\Gamma-$cuspidal parabolic subgroups in $G$ and $N(P)$ denotes the unipotent radical of $P$.

Let $\mathfrak{g}=\mathfrak{so}(d,1)$ and $\mathfrak{k} \cong \mathfrak{so}(d)$ be the Lie algebras of $G$ and $K$ respectively and let $\mathfrak{g} = \mathfrak{k} \oplus \mathfrak{p}$ be the decomposition of $\mathfrak{g}$ induced by the Cartan involution $\theta$. By $\mathfrak{a}$ we denote the (one-dimensional) maximal abelian subspace of $\mathfrak{p}$ and let $M\cong SO(d-1)$ be the centralizer of $A=\exp(\mathfrak{a})$. The positive restricted root of $(\mathfrak{g},\mathfrak{a})$ with the normalized Cartan-Killing norm equal to one is denoted by $\beta$, whereas $\mathfrak{n}$ denotes the positive root space of $\beta$.

Every hyperbolic element $\gamma\in\Gamma$ has the form $m_\gamma a_{\gamma} \in MA^{+}$, where $A^+= \{a\in A ~|~ a = \exp(tH), ~H \in \mathfrak{a},~ t>0\}$. The length of a closed geodesic $C_{\gamma}$ determined by $\gamma$ is denoted by $l(C_{\gamma})$, whereas $j(\gamma)$ is a positive integer such that $\gamma=\gamma_0^{j(\gamma)}$, for a primitive $\gamma_0$.

In the case when $d=2n$ or $d=2n+1$ and $l\neq n$, we denote by $\sigma_l$ the irreducible fundamental representation of $M=SO(d-1)$ acting on $\bigwedge^l (\CC^{d-1})$ and by $d(\sigma_l)$ the dimension of the representation space of $\sigma_l$. In the case when $d=2n+1$ and $l=n$, we denote by $\sigma_n^{\pm}$ the two irreducible half spin representations acting on $\bigwedge^n (\CC^{2n})$ and by $d(\sigma_n)$ the dimension of the representation space of $\sigma_n^{\pm}$.

Going further, let $\chi$ be a unitary representation of $\Gamma$ on a Hermitian vector space $V_{\chi}$ of a finite dimension $\dim V_{\chi}$. For $P \in \mathfrak{P}_{\Gamma}$, let $V_P \subseteq V_{\chi}$ denote the maximal subspace of $V_{\chi}$ on which $\left .\chi \right |_{\Gamma_P}$ acts trivially, and put $d_P(\chi)= \dim V_P$.

For $0\leq k \leq \lfloor(d-1)/2 \rfloor$, the Selberg zeta function $Z_{\chi}(\sigma_k, s)$, for $\Re(s)>d-1$ is defined by the absolutely convergent product

\begin{align*}
Z_{\chi}(\sigma_k, s)= \prod _{\gamma_{\Gamma} \in \mathrm{P}\Gamma_{\hyp}} \prod_{k=0}^{\infty} \det \left( \mathrm{Id} - \overline{\sigma_k(m_{\gamma})} \otimes \chi(\gamma) \otimes S^k (\mathrm{Ad}\left.(m_{\gamma} a_{\gamma})\right|_{\overline{\mathfrak{n}}})e^{-s l(C_{\gamma})} \right)
\end{align*}
\begin{align*}
= \exp\left(  -  \sum_{\gamma \in\Gamma_{\hyp}} \mathrm{Tr}\chi(\gamma) j(\gamma)^{-1} D(\gamma)^{-1} \overline{\mathrm{Tr} \sigma_k(m_{\gamma})} \exp(- (s-(d-1/2))l(C_{\gamma}) )\right).
\end{align*}
Here, $\mathrm{P}\Gamma_{\hyp}$ denotes the set of $\Gamma-$conjugacy classes of primitive hyperbolic elements of $\Gamma$, $\overline{\mathfrak{n}} = \theta\mathfrak{n}$, $S^k$ is the $k-$th symmetric power of $\mathrm{Ad}\left.(m_{\gamma} a_{\gamma})\right|_{\overline{\mathfrak{n}}}$ and
$$
D(\gamma)= D(m_\gamma a_{\gamma}) = \exp\left(\frac{d-1}{2} l(C_{\gamma}) \right) \left| \det (\mathrm{Ad}(m_{\gamma} a_{\gamma})^{-1} - \mathrm{Id}|_{\mathfrak{n}} ) \right|.
$$
Finally, we denote by $C_{\chi}^k(\sigma_k,s)$ the hyperbolic scattering operator (i.e. hyperbolic scattering matrix) defined in section 2.6. of \cite{GonPark}, see also \cite{Kelmer15} where the divisor of $C_{\chi}^k(\sigma_k,s)$ is derived.

With the above notation, we state the divisor of $Z_{\chi}(\sigma_k, s)$, for $0\leq k \leq \lfloor(d-1)/2 \rfloor$. The Selberg zeta function $Z_{\chi}(\sigma_k, s)$, defined for $\Re(s)>d-1,$ has a meromorphic continuation to $\CC$ with non-trivial zeros at
\begin{itemize}
\item $s=(d-1)/2 \pm i \lambda_j(k)$ of order $m_j(k)$, where $\lambda_j(k)^2 + ((d-1)/2-k)^2$ is an eigenvalue of the Hodge Laplacian $\Delta_k$ acting on the $L^2-$space of $k-$forms twisted by $\chi$, with multiplicity $m_j(k)$
    \item $s=(d-1)/2+q$ of order $d(\sigma_k)b$, where $s=q$ is a pole of $\det C_{\chi}^{k} (\sigma_k,s)$ with $\Re(q)<0$,  of order $b.$
\end{itemize}

The location of poles and trivial zeros depends on the parity of $d$. When $d=2n+1$ is odd, $Z_{\chi}(\sigma_k, s)$ has no trivial zeros, and has poles at
\begin{itemize}
\item $s=k$ of order $d_c(\chi)e(d,k) $,  where $d_c(\chi) = \sum_{P\in\mathfrak{P}_{\Gamma}} d_P(\chi)$ and $e(d,k)= (-1)^k\left( \sum_{j=0}^k (-1)^j d (\sigma_j)\right) \geq 0$
\item $s=n$ of order $\frac{1}{2} d(\sigma_k) \mathrm{Tr}\left(n(\sigma_k) \mathrm{Id} - C_{\chi}^{k} (\sigma_k,0) \right) =\frac{1}{2} d(\sigma_k) a_k $
    \item $s= n - q_j$ of order $d(\sigma_k)b_j$, where $s=q_j$ is a pole of $\det C_{\chi}^{k} (\sigma_k,s)$ with $0<q_j<n$,  of order $b_j$, for finitely many $j$, say $j=1,...,K$
    \item $s= n-l$, for $l\in\NN \setminus \{n-k\}$ of order $d_c(\chi)d(\sigma_k)$, when $k\neq n$ and in case $k=n$, order is $2d_c(\chi)d(\sigma_k).$
\end{itemize}

When $d=2n$ is even, $Z_{\chi}(\sigma_k, s)$ has poles at
\begin{itemize}
    \item $s= n -\frac{1}{2} - q_j$ of order $d(\sigma_k)b_j$, where $s=q_j$ is a pole of $\det C_{\chi}^{k} (\sigma_k,s)$ with $0<q_j<n$,  of order $b_j$, for finitely many $j$
    \item $s= n-\frac{1}{2}-l$, for $l\in\NN$ of order $d_c(\chi)d(\sigma_k)$
\end{itemize}
and possesses zeros or poles (according to the sign of their orders) at
\begin{itemize}
    \item $s=k$ of order $d_c(\chi)d(d,k) + (-1)^{k+1} \dim V_{\chi} E(X_{\Gamma}) $,  where $d(d,k)= d(\sigma_k)-e(d,k)\geq 0$
    \item $s= -l$, $l\in\NN$ of order $-\dim V_{\chi} E(X_{\Gamma}) \left[ \binom{2n+l-1}{l+k}\binom{l+k-1}{k} + \binom{2n+l-1}{k}\binom{2n+l-k-2}{l-1}  \right]$, where $E(X_{\Gamma})$ is the Euler characteristic of $X_{\Gamma}$.
\end{itemize}

We adopt the convention that in the case when the locations of  zeros or poles coincide, their orders are added.

Obviously, the function $Z_{\chi}(\sigma_k, s)$, for $0\leq k \leq \lfloor(d-1)/2 \rfloor$ is a zeta-type function.
Assume that $0\leq k \leq \lfloor(d-1)/2 \rfloor$ is fixed and denote by $\Z^\text{NT}_{f}(s,z)$ the superzeta function associated to the set of its non-trivial zeros, i.e. for $\Re(s)>d-1$ and $z\in\CC$ such that $z-s \notin \RR^-$, for all non-trivial (NT) zeros $s$ of $Z_{\chi}(\sigma_k, s)$ we have
\begin{multline*}
\Z^\text{NT}_{\chi, \sigma_k}(s,z) = \sum_{j}\left(\frac{m_j(k)}{(z-(d-1)/2 - i\lambda_j(k))^s} + \frac{m_j(k)}{(z-(d-1)/2 + i\lambda_j(k))^s} \right)\\ +\sum_{\Re(q)<0} \frac{d(\sigma_k)b}{(z-(d-1)/2 - q)^s}.
\end{multline*}

To define the superzeta function associated to the set of trivial zeros and poles we need to consider two cases, depending on the parity of $d$.

In case when $d$ is odd, the superzeta function associated to the set of poles (P) of $Z_{\chi}(\sigma_k, s)$ is defined for $\Re(s)>d-1$ and $z\in\CC$ such that $z-s \notin \RR^-$, for all poles $s$ of $Z_{\chi}(\sigma_k, s)$ by
\begin{align*}
\Z^\text{P}_{\chi, \sigma_k}(s,z)& = \frac{d_c(\chi)e(d,k)}{(z-k)^s} + \frac{1}{2} d(\sigma_k) \mathrm{Tr}\left(n(\sigma_k) \mathrm{Id} - C_{\chi}^{k} (\sigma_k,0) \right) (z-n)^{-s} \\
&+  \sum_{j=1}^{K}\frac{d(\sigma_k)b_j}{(z-(n-q_j))^s} + (1+\delta_{kn})d_c(\chi)d(\sigma_k) \zeta_H(s, z-n+1) \\&- (1-\delta_{kn})d_c(\chi)d(\sigma_k) (z-k)^{-s} .
\end{align*}
Note that $\zeta_H$ is the Hurwitz zeta function $\zeta_H(s,z) = \sum_{l=0}^{\infty}(z+l)^{-s};$ and 
$\delta_{kn}$ denotes the Kronecker delta function.

In the case when $d=2n+1$ is odd, we have the following corollary of Theorem~\ref{repZ1Thm}.

\begin{cor} \label{cor odd n}
For $d=2n+1$, $0\leq k \leq n$ and $z\in\CC$ such that $z-s \notin \RR^-$, for any non-trivial zero $s$ of $Z_{\chi}(\sigma_k, s)$ the superzeta function $\Z^\mathrm{NT}_{\chi, \sigma_k}(s,z)$ associated to the set of non-trivial zeros of $Z_{\chi}(\sigma_k, s)$, initially defined for $\Re(s)>2n+1,$ possesses a meromorphic continuation to the whole complex $s$-plane, with only one simple pole at $s=1$ with corresponding residue equal to  $\beta(k,n,\chi)=(1+\delta_{kn})d_c(\chi)d(\sigma_k)$.
Moreover, the zeta regularized product associated to this zeta function is given by
\begin{align} \notag
\exp\left(-\left. \frac{d}{ds}\Z^\mathrm{NT}_{\chi, \sigma_k}(s,z) \right|_{s=0} \right) &= (z-k)^{\alpha(k,n,\chi)}(z-n)^{\frac{1}{2}d(\sigma_k)a_k}\prod_{j=1}^{K}(z-(n-q_j))^{d(\sigma_k)b_j} \cdot \\ &\cdot \left( \frac{\Gamma(z-n+1)}{\sqrt{2\pi}}\right)^{-\beta(k,n,\chi)}Z_{\chi}(\sigma_k, z) \label{regulProd Sel zeta} ,
\end{align}
where $\alpha(k,n,\chi) = d_c(\chi)(e(d,k) - (1-\delta_{kn}) d(\sigma_k))$.

\end{cor}
\begin{proof}
Remark \ref{remark}, Theorem \ref{repZ1Thm} together with definition of the superzeta function associated to the set of poles of $Z_{\chi}(\sigma_k, s)$ yield that
\begin{align} \notag
\Z^\text{NT}_{\chi, \sigma_k}(s,z)&= d_c(\chi)(e(d,k) - (1-\delta_{kn}) d(\sigma_k)) (z-k)^{-s} \\ \notag&+ d(\sigma_k)(a_k/2 - (1+\delta_{kn})d_c(\chi))(z-n)^{-s}\\ \label{Z non-triv}
&+ \sum_{j=1}^{K}\frac{d(\sigma_k)b_j}{(z-(n-q_j))^s} + (1+\delta_{kn})d_c(\chi)d(\sigma_k)\zeta_H(s, z-n) \\ \notag
&+ \frac{\sin \pi s}{\pi }\int_{0}^{\infty }\left( \frac{Z_{\chi} (\sigma_k, z+y)^{\prime }}{Z_{\chi}(\sigma_k, z+y)}\right) y^{-s}dy,
\end{align}\notag
and the above equation is valid for all admissible $z$, i.e. $z\in\CC$ such that $z-s \notin \RR^-$, for all non-trivial zeros $s$ of $Z_{\chi}(\sigma_k, s)$.

For a fixed, admissible $z$, all terms on the right-hand side of \eqref{Z non-triv}, except the Hurwitz zeta function are holomorphic in $s$. Therefore, for a fixed, admissible $z$, the function $\Z^\text{P}_{\chi, \sigma_k}(s,z)$ possesses a meromorphic continuation to the whole complex $s-$plane, with only one pole at $s=1$ and corresponding residue $(1+\delta_{kn})d_c(\chi)d(\sigma_k)$ stemming from the factor of the Hurwitz zeta function.

Differentiating equation \eqref{Z non-triv} with respect to the $s$ variable, inserting $s=0,$ and employing equation $\frac{d}{ds}\zeta_H(s,z)|_{s=0} = \log \Gamma(z) - \log \sqrt{2\pi} $, we get
\eqref{regulProd Sel zeta}.
\end{proof}

In the case when $d=2n$ is even, the superzeta function constructed over poles/trivial zeros is somewhat more complicated. First, we introduce multiple Hurwitz zeta function and multiple gamma functions.

For a positive integer $m$, the multiple Hurwitz zeta function is defined for $\Re(s)>m$ and $z\in\CC \setminus\{0,-1,-2,...\}$ by the absolutely convergent series
$$
\zeta_{m}(s,z) = \sum_{l=0}^{\infty} \binom{m+l-1}{l}(z+l)^{-s},
$$
which possesses a meromorphic continuation to the whole complex $s-$plane with simple poles at $s=1,2,...,m$ (see \cite{ChSr11}, p. 505). In \cite{ChSr11} it is also proved that the multiple Hurwitz zeta function is related to Hurwitz zeta function via the formula
$$
\zeta_m(s,z)= \sum_{j=0}^{m-1} p_{m,j}(z)\zeta_H(s-j,z),
$$
where $p_{m,j}(z)$ are polynomials in $z$ defined by
$$
p_{m,j}(z)= \frac{1}{(m-1)!} \sum_{l=j}^{m-1}(-1)^{m+1-j}\binom{l}{j}s(n,l+1)z^{l-j}
$$
and $s(n,l)$ is the Stirling number of the first kind, i.e. the coefficient in the expansion
$$
(z)_m= z(z+1)\cdots (z+m-1)= \sum_{l=0}^{m}(-1)^{m+l}s(m,l)z^l
$$
of the Pochhammer symbol $(z)_m$.

Since the Hurwitz zeta function $\zeta_H(s,z)$ possesses only one simple pole at $s=1$, with residue one, it is immediate that residues of $\zeta_{m}(s,z)$ at $s=k \in\{1,...,m\}$ are
$$
\mathrm{Res}_{s=k} \zeta_m(s,z)=p_{m,k-1}(z).
$$
The multiple gamma function $\Gamma_m$ is defined by
$$
\Gamma_m(z)= \exp\left( \left .\frac{d}{ds} \zeta_m(s,z) \right|_{s=0} \right).
$$

\begin{cor} \label{cor even n}
For $d=2n$, $0\leq k \leq n-1$ and $z\in\CC$ such that $z-s \notin \RR^-$, for any non-trivial zero $s$ of $Z_{\chi}(\sigma_k, s)$ the superzeta function $\Z^\mathrm{NT}_{\chi, \sigma_k}(s,z)$ associated to the set of non-trivial zeros of $Z_{\chi}(\sigma_k, s)$, initially defined for $\Re(s)>2n,$ possesses a meromorphic continuation to the whole complex $s$-plane, with simple poles at $s_r=r$, for $r\in\{1,...,2n\}$ with corresponding residues equal to
\begin{multline} \label{residues}
\mathrm{Res}_{s=r} \Z^\mathrm{NT}_{\chi, \sigma_k}(z, s) = \delta_{1r}d_c(\chi)d(\sigma_k)  \\ +\dim V_{\chi} E(X_{\Gamma})\sum_{j=0}^{k} (-1)^j\binom{2n}{k-j}\left( p_{2n,r-1} (z-j) + p_{2n,r-1}( z+j+1)\right).
\end{multline}
Moreover, the zeta regularized product associated to this superzeta function is given by
\begin{multline}\label{regulProd Sel zeta even}
\exp\left(-\left. \frac{d}{ds}\Z^\mathrm{NT}_{\chi, \sigma_k}(s,z) \right|_{s=0} \right) = Z_{\chi}(\sigma_k, z)\cdot (z-k)^{-\gamma(k,n,\chi)} \left( \prod_{j=1}^{K}(z-n-\frac{1}{2}-q_j)^{d(\sigma_k)b_j} \right) \\ \cdot \left( \frac{\Gamma(z-n+\frac{3}{2})}{\sqrt{2\pi}}\right)^{-d_c(\chi) d(\sigma_k)}\cdot\left( \prod_{j=0}^{K}  \left(\Gamma_{2n}(z-j) \Gamma_{2n}(z+j+1) \right) ^{(-1)^j \binom{2n}{k-j}} \right)^{\dim V_{\Gamma}E(X_{\Gamma}) },
\end{multline}
where $\gamma(k,n,\chi)=d_c(\chi)d(d,k) + (-1)^{k+1}(1-\dim V_{\chi} E(X_{\Gamma}))$.
\end{cor}
\begin{proof}
From definition of divisor of $Z_{\chi}(\sigma_k, s)$, for all $z\in\CC$ such that $z-s \notin \RR^-$ for all $w$ which belong to the set of non-trivial zeros or poles of $Z_{\chi}(\sigma_k, s)$ we define for $\Re(s)>2n$ (following notation set in Remark 5)
\begin{multline*}
\Z^\text{P}_{\chi, \sigma_k}(z, s)-\Z^\text{T}_{\chi, \sigma_k}(s,z) = - \frac{d_c(\chi)d(d,k) + (-1)^{k+1}\dim V_{\Gamma}E(X_{\Gamma})}{(z-k)^{s}}\\ +\sum_{j=1}^{K}\frac{d(\sigma_k)b_j}{(z-(n-1/2-q_j))^s}+ d_c(\chi)d(\sigma_k) \zeta_H(s, z-n+3/2)\\
+ \dim V_{\chi} E(X_{\Gamma}) \sum_{l\in\NN}\left[ \binom{2n+l-1}{l+k}\binom{l+k-1}{k} + \binom{2n+l-1}{k}\binom{2n+l-k-2}{l-1}  \right](z+l)^{-s}
\end{multline*}

\noindent
Next, we employ the following transformation formulas for the binomial coefficients:
$$
\sum_{j=0}^{k-m}(-1)^{j}\binom{2n}{k-m-j}\binom{2n+j-1}{j}=0, \text{   for   } m\in\{0,...,k-1\};
$$
$$
\sum_{j=0}^{k}(-1)^{j}\binom{2n}{k-j}\binom{2n+m+j-1}{m+j}= \binom{2n+m-1}{m+k}\binom{m+k-1}{k}
$$
and
$$
\sum_{j=0}^{\min\{k, m-1\}}(-1)^{j}\binom{2n}{k-j}\binom{2n+m-j-2}{m-j-1}= \binom{2n+m-1}{k}\binom{2n+m-k-2}{m-1},
$$
for positive integers $m$ to deduce, by equating coefficients of $(z+m)^{-s}$, for $m-1 \geq k,$ that
\begin{multline*}
\sum_{l\in\NN}\left[ \binom{2n+l-1}{l+k}\binom{l+k-1}{k} + \binom{2n+l-1}{k}\binom{2n+l-k-2}{l-1}  \right](z+l)^{-s}= \\= \sum_{j=0}^{k} (-1)^j\binom{2n}{k-j}\left( \zeta_{2n} (s, z-j) + \zeta_{2n}(s, z+j+1)\right) + (-1)^{k+1}(z-k)^{-s}.
\end{multline*}

Therefore, application of Theorem \ref{repZ1Thm} yields the equation
\begin{multline*}
\Z^\text{NT}_{\chi, \sigma_k}(s,z) =   \sum_{j=1}^{K}\frac{d(\sigma_k)b_j}{(z-(n-1/2-q_j))^s}+ d_c(\chi)d(\sigma_k) \zeta_H(s, z-n+3/2) + \\  +\dim V_{\chi} E(X_{\Gamma})\sum_{j=0}^{k} (-1)^j\binom{2n}{k-j}\left( \zeta_{2n} (s, z-j) + \zeta_{2n}(s, z+j+1)\right) +\\+ \frac{(-1)^{k+1}(\dim V_{\chi} E(X_{\Gamma})-1)-d_c(\chi)d(d,k)}{(z-k)^s}+ \frac{\sin \pi s}{\pi }\int_{0}^{\infty }\left( \frac{Z_{\chi} (\sigma_k, z+y)^{\prime }}{Z_{\chi}(\sigma_k, z+y)}\right) y^{-s}dy.
\end{multline*}
For admissible values of $z$, function on the right-hand side of the above equation is meromorphic in $s$ with simple poles at $s\in\{1,...,2n\}$ and corresponding residues at $s=r\in\{1,...,2n\}$ given by equation \eqref{residues}, which proves the first part.

By taking derivatives with respect to the $s-$variable, inserting $s=0$ and
using the definition of multiple gamma function, we immediately deduce \eqref{regulProd Sel zeta even}.
\end{proof}

\begin{rem}  \rm
The zeta regularized product $D_{\chi,\sigma_k}(z)=\exp\left(-\left. \frac{d}{ds}\Z^\mathrm{NT}_{\chi, \sigma_k}(s,z) \right|_{s=0} \right)$ associated to the Selberg zeta function is evaluated in Corollaries \ref{cor odd n} and \ref{cor even n} for admissible values of $z$, i.e. $z\in\CC$ such that $z-s \notin \RR^-$, for any non-trivial zero $s$ of $Z_{\chi}(\sigma_k, s)$. The right-hand side of \eqref{regulProd Sel zeta} and \eqref{regulProd Sel zeta even} provides meromorphic continuation of $D_{\chi,\sigma_k}(z)$ to the cut plane $\CC \setminus(-\infty,n]$.
\end{rem}

%Proposition is good to apply when one knows exact factors of the functional equation and is able to derive asymptotic behavior of the factor of functional equation, which is %essentially enough for application of Proposition \ref{prop: Voros cont.}.

%In case one knows trivial zeros and poles of the zeta function and factor of the functional equation is "ugly" - it is more convenient to apply theorem \ref{repZ1Thm}.

\bibliographystyle{amsplain}

\end{document}